\documentclass[12pt]{amsart}
\usepackage{amsmath,amssymb,amsfonts,a4wide,color,mathtools,url}
\usepackage{amsrefs}
\usepackage{cite}
\usepackage[utf8]{inputenc} 
\usepackage{bm} 

\let\phi\varphi
\let\epsilon\varepsilon
\let\setminus\smallsetminus
\let\emptyset\varnothing

\let\leq\leqslant
\let\geq\geqslant

\newtheorem{Thm}{Theorem}[section]
\newtheorem{Prop}[Thm]{Proposition}
\newtheorem{Lem}[Thm]{Lemma}

\numberwithin{equation}{section}

\def\qed{{\hskip0pt\unskip\unskip\nobreak\hfil\penalty50
          \hskip1em\hbox{}\nobreak\hfil
           {$\square$}
          \parfillskip=0pt\finalhyphendemerits=0
          \par}\medskip}

               {\noindent{\bf Proof.}\ }
               {\qed}

               {\noindent{\bf Proof of #1.}\ }
               {\qed}

\begin{document}


\title{On the continued fraction expansion of absolutely normal numbers}

\author[A.-M. Scheerer]{Adrian-Maria Scheerer}
\address[A.-M. Scheerer]{Institute of Analysis and  Number
  Theory\\Graz University of Technology\\A-8010 Graz, Austria}
\email{scheerer@math.tugraz.at}

%

\begin{abstract}
We construct an absolutely normal number whose continued fraction expansion is normal in the sense that it contains all finite patterns of partial quotients with the expected asymptotic frequency as given by the Gauss-Kuzmin measure. The construction is based on ideas of Sierpinski and uses a large deviations theorem for sums of mixing random variables. 
\end{abstract}

\date{\today}

\maketitle


\section{Introduction}

Let $x$ be a real number in the unit interval $[0,1)$ and let $b\geq 2$ be a positive integer. Consider the maps $T_b : [0,1) \to [0,1), x \mapsto bx \bmod 1$ and the Gauss map $T_G:[0,1) \to [0,1)$ defined by $T_G(x) = \frac{1}{x} \bmod 1$ if $x>0$ and $T_G(0)=0$. Then $x$ is called \emph{normal to base $b$}, if for all real numbers $0\leq \alpha < \beta < 1$,
\begin{equation}\label{base_b_norm_def}
\frac{1}{n}\sum_{i= 0}^{n-1} \chi_{[\alpha,\beta)}(T_b^i(x)) \rightarrow \beta - \alpha
\end{equation}
holds, as $n$ tends to infinity. Here, $\chi_A$ is the characteristic function of the set $A$. $x$ is called \emph{continued fraction normal}, if for all $0\leq \alpha < \beta < 1$,
\begin{equation}\label{cf_norm_def}
\frac{1}{n}\sum_{i= 0}^{n-1} \chi_{[\alpha,\beta)}(T_G^i(x)) \rightarrow \mu_G([\alpha,\beta)),
\end{equation}
where $\mu_G$ is the Gauss-Kuzmin measure on $[0,1)$ given by
\begin{equation}\label{mu_G}
\mu_G(A) = \frac{1}{\log 2} \int_A \frac{1}{1+x}dx
\end{equation}
for any Borel set $A$. 

It is in fact enough to consider in definitions~\eqref{base_b_norm_def} and~\eqref{cf_norm_def} so-called cylinder sets. These are intervals all of whose elements share the same beginning in their base-$b$ expansion or continued fraction expansion, respectively. This way we recover the more familiar definition of normality via the expected behaviour of the asymptotic frequencies of all finite digit patterns.

The maps $T_b$ are invariant and ergodic with respect to the Lebesgue measure and the Gauss map $T_G$ is invariant and ergodic with respect to $\mu_G$. An application of the point-wise ergodic theorem thus shows that with respect to Lebesgue measure almost all real numbers in the unit-interval are simultaneously normal to all integer bases $b\geq 2$ (such numbers are called \emph{absolutely normal}) and continued fraction normal. The aim of this note is to exhibit an example of such a number by means of describing its binary expansion one digit after the other using a recursive construction.

Our construction is based on ideas of Sierpinski~\cite{sierpinski1917:borel_elementaire} and Becher and Figueira \cite{becher_figueira:2002} and can be described as follows. We consider a suitable large subset $\Omega$ of $[0,1)$ as our ambient set. This set contains all real numbers whose partial quotients grow at a controlled rate (see Section~\ref{Sec_Omega}). We wish to exclude from this set the set of all non-normal numbers and do so by collecting these numbers in a set $E$. This set will in fact have positive but small measure. Part of the proof is showing that this set is in fact `small'. The corresponding calculations are carried out in Sections~\ref{Sec_large_dev} and~\ref{Sec_set_E}. The main new ingredient is the use of a large deviations theorem for sums of mixing random variables to control deviations in~\eqref{cf_norm_def}. In Section~\ref{Sec_algorithm} we compute the binary expansion of a number $\nu$ in $\Omega\setminus E$. This is done starting with the interval $[0,1)$ and subsequently considering recursively both halves of the preceding interval and deciding which half is `best', i.e. contains more of $\Omega\setminus E$. To make this construction computable, we actually work with finitary versions of $\Omega$ and $E$ at the cost of a small but controllable error. Finally, in Theorem~\ref{BIGTHEOREM} we show that $\nu$ is computable and indeed simultaneously normal to every integer base $b\geq 2$, as well as continued fraction normal. Section~\ref{sec_set_theoretic_lemmas} contains some ancillary set-theoretic lemmas used in Section~\ref{Sec_algorithm}.

Normal numbers originated in the work of Borel from 1909~\cite{borel1909}. The reader is best advised to consult the books~\cite{kuipers_niederreiter,drmota_tichy,bugeaud2012distribution} for an introduction and concise treatment of the subject.


Although there exist many constructions of normal numbers (to a single base), no easy construction of a number normal to two multiplicatively independent bases is known. However, recently constructions of absolutely normal numbers via recursively formulated algorithms have received much interest. If `easy' is interpreted from a computational viewpoint, the problem has been solved by Becher, Heiber and Slaman~\cite{becher_heiber_slaman2013:polynomial_time_algorithm}, who gave a polynomial time algorithm for computing the digits (to some base) of an absolutely normal number, and very recently by Lutz and Mayordomo~\cite{lutz_mayordomo_nearly_linear_time}, who gave a nearly linear time algorithm. Other polynomial time algorithms have been announced in~\cite{figueira_nies} and~\cite{mayordomo}. Further constructions of absolutely normal numbers include works by Lebesgue~\cite{MR1504765}, Turing~\cite{turing1992:collected_works} (see also~\cite{becher_figueira_picchi2007:turing_unpublished}), Schmidt~\cite{schmidt1961:uber_die_normalitat} (see also~\cite{scheerer2015:schmidt}) and Levin~\cite{levin1979:absolutely_normal}  (see also~\cite{alvarez_becher2015:levin}).

Explicit examples of continued fraction normal numbers have been given by Postnikov and Pyatecki{\u\i}~\cite{MR0101857}, Adler, Keane and Smorodinsky~\cite{ADLER198195},  Madritsch and Mance~\cite{Madritsch2016} and Vandehey~\cite{Vandehey2016424} by concatenating suitable strings of partial quotients. It remained an open problem to construct an absolutely normal number that is continued fraction normal (see \cite[Ch. 10]{bugeaud2012distribution} and~\cite{queffelec}).

Since our approach is based on the construction of Sierpinski \cite{sierpinski1917:borel_elementaire} and of Becher and Figueira \cite{becher_figueira:2002}, it is expected to have double exponential complexity (see \cite{scheerer2015:schmidt}). In view of the above mentioned much faster algorithms, we have thus refrained from analyzing its complexity.

\medskip
%

We call a real number \emph{computable}, if its binary expansion is computable in a naive sense; i.e. if there is a deterministic algorithm, only using addition, subtraction, multiplication, division and comparison, that outputs the binary expansion of this number one digit after the other, requiring to carry out only finitely many operations for each digit.

\section{Large Deviation Estimates}
\label{Sec_large_dev}

\subsection{Non-normal numbers for integer bases}

Let $b\geq 2$ be an integer. A word $\bm\omega = \omega_1 \ldots \omega_n$ of $n$ digits $0 \leq \omega_i \leq b-1$, $1 \leq i \leq n$, is called \emph{$(\epsilon,1)$-normal of length $n$}, or $\epsilon$-simply normal, if for each digit $0 \leq d \leq b-1$, 
$$n \frac{1}{b}(1-\epsilon) < N(d,\bm\omega) < n \frac{1}{b}(1+\epsilon),$$ where $N(d,\bm\omega)$ is the number of $i$, $1 \leq i \leq n$, such that $\omega_i = d$. Let $E_b(\epsilon, n)$ be the set of all real numbers $x \in [0,1)$ such that the first $n$ digits of the base-$b$ expansion of $x$ form an $(\epsilon,1)$-normal word of length $n$. Denote the complement of $E_b(\epsilon,n)$ in $[0,1)$ by $E_b^c(\epsilon,n)$.

Fix a digit $d$, $0 \leq d \leq b-1$, and consider the random variables $X_i : [0,1) \rightarrow \mathbb{R}$, for \mbox{$1\leq i \leq n$}, defined by $X_i(x) = 1$ if the $i$-th digit in the $b$-ary expansion of $x$ equals $d$, and $X_i(x) = 0$ otherwise. The $X_i$ are independent and have expectation $\frac{1}{b}$. Let $S_n = X_1 + \ldots + X_n$. Then Hoeffding's inequality for the sum of $n$ i.i.d. random variables bounded by $0$ and $1$ yields
\begin{equation*}
\mathbb{P}\left(\left\vert \frac{S_n}{n} - \mathbb{E}\left(\frac{S_n}{n}\right)\right\vert \geq t\right) \leq 2 \exp(-2nt^2).
\end{equation*}
In our case, the probability measure is the Lebesgue measure $\lambda$ on the unit interval. With $t = \frac{\epsilon}{b}$ we obtain
\begin{equation*}
\lambda\left( \left\{ x \in [0,1) : \left\vert \frac{1}{n} \sharp \left\{1\leq i \leq n : X_i(x) = d \right\} - \frac{1}{b} \right\vert \geq \frac{\epsilon}{b}\right\}\right) \leq 2 \exp\left(-\frac{2 \epsilon^2}{b^2} n\right).
\end{equation*}
Hence, for the set of non-$(\epsilon,1)$-normal numbers of length $n$,
\begin{equation*}
\lambda\left(E^c_b(\epsilon,n)\right) \leq 2b \exp\left(-\frac{2 \epsilon^2}{b^2} n\right).
\end{equation*}


If one is interested in $(\epsilon,k)$-normality as introduced by Besicovitch~\cite{besicovitch1935:epsilon}, where one wants to control all combinations of $k$ digits, it is possible to obtain an analogous result by replacing Hoeffding's inequality with Theorem~5 in~\cite{siegel:towards_a_usable_theory_of_chernoff}.

\subsection{Non-normal numbers for continued fractions}

Any real number $x \in [0,1)$ has a continued fraction expansion, denoted as $x = [0; a_1(x), a_2(x), \ldots]$ where the $a_i(x)$ are positive integers. This expansion is finite if and only if $x$ is rational.  For $i \geq 1$, the $a_i(x)$ are obtained by $a_i(x) = \lfloor 1/ T_G^{i-1}(x) \rfloor$, unless $T_G^{i-1} (x)$ is non-zero. If $x$ is understood, we will simply write $a_i$.

Let $A$ be a Borel subset of $[0,1)$ and denote by $\mu_G(A)$ its Gauss measure as introduced in~\eqref{mu_G}. For the Lebesgue measure $\lambda$ we have
\begin{equation*}
\frac{1}{2 \log 2} \lambda(A) \leq \mu_G(A) \leq \frac{1}{\log 2} \lambda(A).
\end{equation*}
Note that for a positive integer $D$, $\lambda(\{x \in [0,1) : a_1(x) \leq D \}) = \frac{D}{D+1}$.

\medskip

Let $\epsilon > 0$ and let $k, D, n$ be positive integers. 
A word $\bm\omega = \omega_1 \ldots \omega_n$ of length $n$ of digits $\omega_i \in \{ 1, \ldots D\}$ will be called \emph{$(\epsilon, k, D, n)$-continued-fraction-normal}, if for all words $\bm d=d_1 \ldots d_k$ of length $k$ of digits $d_j \in \{1, \ldots, D\}$,
\begin{equation}\label{ineq-cf-norm}
(n-k+1) \mu_G(\Delta_{\bm d})(1 - \epsilon) < N({\bm d},\bm\omega) < (n-k+1) \mu_G(\Delta_{\bm d})(1 + \epsilon)
\end{equation}
holds, where $N({\bm d},\bm\omega)$ is the number of $i$, $1 \leq i \leq n-k+1$, such that $\omega_i \ldots \omega_{i+k-1} = d_1 \ldots d_k$ and where $\Delta_{\bm d}$ is the set of all real numbers in $[0,1)$ whose continued fraction expansion coincides on the first $k$ digits with ${\bm d}$.


The set of real numbers $x \in [0,1)$ whose first $n$ partial quotients form a word that is $(\epsilon, k, D, n)$-CF-normal will be denoted by $E_{\text{CF}}(\epsilon,k,D,n)$. We denote its complement in $[0,1)$ by $E^c_{\text{CF}}(\epsilon, k, D, n)$.
We also require a notation for the set of $x \in [0,1)$, where the number of occurrences of only one specific ${\bm d}$ of length $k$ of digits in $\{1, \ldots, D\}$ satisfies~\eqref{ineq-cf-norm}. This set will be denoted by $E_{\text{CF}}(\epsilon, {\bm d}, D, n)$.
Similarly, we introduce the sets $E_{\text{CF}}(\epsilon,k,n)$ and $E_{\text{CF}}(\epsilon,{\bm d},n)$ without restriction on the partial quotients. Complements will be relative to $[0,1)$.

\medskip

Fix a word ${\bm d}$ of length $k$ composed from positive integers. For $i \geq 1$ we have the random variables $a_i : [0,1) \rightarrow \mathbb{R}$ and derived random variables $X_i : [0,1) \rightarrow \mathbb{R}$. The $a_i$ are defined by $a_i(x) = a_i$ when the continued fraction expansion of $x$ is $x = [0; a_1, a_2, \ldots, a_i, \ldots]$. The $X_i$ are defined to be $1 - \mu_i$ if the string ${\bm d}$ appears in the continued fraction expansion of $x$ starting at $a_i$, and $ - \mu_i$ if not. 
The numbers $\mu_i$ are chosen such that $E[X_i] = 0$. 

A sequence $(X_i)_{i \geq 1}$ of random variables $X_i : [0,1) \rightarrow \mathbb{R}$ is called \emph{strongly mixing}, if 
\begin{equation}
\alpha(n) := \sup_{l \geq 1} \alpha(M_l, G_{l+n}) \rightarrow 0
\end{equation}
as $n \rightarrow \infty$. Here $M_l = \sigma(X_i, i\leq l)$ and $G_{l+n} = \sigma(X_i, i \geq l+n)$ are the $\sigma$-algebras generated by $X_i$, for $i\leq l,$ and by $X_i$, for $i \geq l+n$. The \emph{$\alpha$-mixing coefficients} $\alpha(M_l, G_{l+n})$ are defined as
\begin{equation*}
\alpha(M_l, G_{l+n}) = \sup_{A \in M_l, B \in G_{l+n}} \vert \mathbb{P}(A\cap B) - \mathbb{P}(A) \mathbb{P}(B) \vert.
\end{equation*}

For an overview of different notions of mixing, see the survey by Bradley~\cite{bradley}. We followed the notation from~\cite{merlevede_peligrad_rio}.

We know the following mixing property of $(a_i)_{i\geq 1}$ with respect to the Gauss map $\mu_G$ on $[0,1)$.

\begin{Thm}[Philipp~\cite{Philipp1988}]\label{mixing_philipp}
The $a_i$ are exponentially strongly mixing. In fact we have for some $0 \leq \rho < 0.8$
\begin{equation}
\vert \mu_G(A\cap B) - \mu_G(A)\mu_G(B) \vert \leq \rho^n \mu_G(A)\mu_G(B)
\end{equation} 
for all $A \in \sigma(a_i, i \leq l)$ and $B \in \sigma(a_i, i \geq n+l)$.
\end{Thm}

The constant $\rho$ has been subject to later improvements (see e.g.~\cite[Prop.~2.3.7]{iosifescu_kraaikamp} and~\cite{kessebohmer_schindler}). We work here with $\rho = 0.8$.

\medskip 
From Theorem~\ref{mixing_philipp} we can derive exponential strong mixing for the random variables $X_i$ with respect to the Gauss measure $\mu_G$. We look at $\vert\mu_G(A\cap B) - \mu_G(A)\mu_G(B)\vert$ where $A \in \sigma(X_1, \ldots, X_l)$ and $B \in \sigma(X_{l+n}, X_{l+n+1}, \ldots)$. Since $\sigma(X_i) = X_i^{-1} \mathcal{B}(\mathbb{R})$ is generated by $\{ \emptyset, [0,1), T_G^{-i}({\bm d}) = a_i^{-1}(\{d_1\}) \cap a_{i+1}^{-1}(\{d_2\}) \cap \ldots \cap a_{i+k-1}^{-1}(\{d_k\}), [0,1) \setminus T_G^{-i}({\bm d}) \}$, we have that $\sigma(X_i) \subset \sigma(a_i, \ldots, a_{i+k-1})$ and hence $\sigma(X_1, \ldots, X_l) \subset \sigma(a_1, \ldots, a_{l+k-1})$. Consequently $\sigma(X_l, X_{l+1}, \ldots ) \subset \sigma(a_l, a_{l+1}, \ldots)$. Thus any mixing coefficient $\alpha(n-k+1)$ for the $a_i$ is a valid mixing coefficient $\alpha(n)$ for the $X_i$, for $n \geq k$. For smaller values of $n$, note that in general $\alpha(n) \leq \frac{1}{4}$.
Hence for all $n\geq 1$ the $X_i$ are strongly mixing with $\alpha$-mixing coefficient $\alpha(n) \leq \exp(-2nc)$  with 
\begin{equation}\label{c}
c = - \frac{\log 0.8}{2k}.
\end{equation}

We have thus shown that there is an explicit $c>0$ such that $(X_i)_{i\geq 1}$ is a sequence of strongly mixing centred real-valued bounded random variables with $\alpha$-mixing coefficient $\alpha(n)$ satisfying $\alpha(n) \leq \exp(-2cn)$. As such, the $X_i$ satisfy the assumptions of the following large deviation theorem.

\begin{Thm}
[Merlev\`{e}de, Peligrad, Rio~
{\cite[Cor.~12]{merlevede_peligrad_rio}}]
\label{large_deviation_mixing}
Let $(X_i)_{i\geq 1}$ be a sequence of centered real-valued random variables bounded by a uniform constant $M$ and with $\alpha(n)$ satisfying $\alpha(n) \leq \exp(-2nc)$ for some $c>0$. Then for all $n\geq 2 \cdot \max(c, 2)$ and $x\geq 0$
\begin{equation}
\mathbb{P}(\vert S_n \vert \geq x) \leq \exp\left(-\frac{x^2}{n (\log n) 4C M^2 + 4 Mx (\min(c,1))^{-1}} \right),
\end{equation}
where $C = 6.2 K + (\frac{1}{c} + \frac{8}{c^2}) + \frac{2}{c \log 2}$, with $K = 1 + 8 \sum_{i \geq 1} \alpha(i)$.
\end{Thm}

Here $S_n$ denotes again the sum $X_1 + X_2 + \ldots + X_n$.


The following theorem is thus a corollary of Theorem~\ref{large_deviation_mixing}.

\begin{Thm}\label{Thm_bound_non_cf_normal_numbers}
Let $\epsilon > 0$ and fix a string ${\bm d}$ of length $k$ of positive integers.
There is $\eta_{\text{CF}}(\epsilon, {\bm d}) > 0$ as specified in~\eqref{eta_cf_usable_bound} such that for $N \geq 2(k+1)$
\begin{equation}\label{bound_non_normal_cf_numbers}
\mu_G(E^c_{\text{CF}}(\epsilon, {\bm d}, N))  \leq \exp\left( - \eta_{\text{CF}}(\epsilon, {\bm d})  \frac{N}{\log N} \right).
\end{equation}
\end{Thm}

\begin{proof}
We set $x = \epsilon \mu_G(\Delta_{\bm d})n$, $\mathbb{P} = \mu_G$, $n = N-k+1$ and $M = 1$. Hence $\vert S_n \vert \geq x$ is the same as $\vert \sum_{i=1}^{N-k+1} X_i - (N-k+1) \mu_G(\Delta_{\bm d}) \vert \geq \epsilon \mu_G(\Delta_{\bm d})(N-k+1)$ which is equivalent to the defining condition of non-$(\epsilon,{\bm d},N)$-continued-fraction-normality from~\eqref{ineq-cf-norm}. We have $0.09 < -\log 0.8 < 0.1$, so for any $k$, $\max(c,1) = 1$ and Theorem~\ref{large_deviation_mixing} can be applied provided $N-k+1 \geq 4$ holds. To estimate the exponent we use $N-k+1 \geq \frac{1}{2} N$, valid for $N \geq 2(k-1)$. The requirement $N\geq 2(k+2)$ meets both conditions on $N$. Thus
\begin{align*}
\mu_G(E^c_{\text{CF}}(\epsilon, {\bm d}, N)) 
& \leq \exp\left( - \frac{(\epsilon \mu_G(\Delta_{\bm d}))^2}{16(C+ (\epsilon \mu_G(\Delta_{\bm d}))/c)} \frac{N}{\log N} \right).
\end{align*}
We wish to simplify the exponent by bounding it from below. This can be achieved by straight-forward calculations, noting that $c\leq 1/20 < 1$ for any $k$, and that $N \mu_G(\Delta_{\bm d})(1+\epsilon) \leq N$, so that  $\mu_G \epsilon \leq 1-\mu_G \leq 1$.
We obtain
\begin{equation}\label{eta_cf_usable_bound}
\eta_{\text{CF}}(\epsilon, {\bm d})  =  \left( \frac{\epsilon \mu_G(\Delta_{\bm d})}{900 k} \right)^2
\end{equation}
as admissible value in~\eqref{bound_non_normal_cf_numbers}.
\end{proof}

%
%
%
%
%

\subsubsection*{Remarks}
The bound obtained in Theorem~\ref{Thm_bound_non_cf_normal_numbers} bounds a set of certain real numbers with a priori no restrictions on their partial quotients. However, since $E^c_\text{CF}(\epsilon, {\bm d}, D, N) \subset E^c_\text{CF}(\epsilon, {\bm d}, N)$ the bound~\eqref{bound_non_normal_cf_numbers} is also valid for this smaller set. Note that $E^c_\text{CF}(\epsilon, {\bm d}, D, N)$ is a union of finitely many intervals with rational endpoints and thus can be computed, as well as its Lebesgue measure (for example by first listing all non-$(\epsilon,{\bm d},D,N)$-normal words).

Vandehey obtained this result in~\cite{Vandehey2016424} with linear decay in $N$. This is not sufficient to ensure convergence as in our application we sum over the error term for all $N$ large enough.

The bound from~\eqref{bound_non_normal_cf_numbers} is valid for the Lebesgue measure of $E^c_\text{CF}(\epsilon, {\bm d}, N)$ with an additional factor of $\frac{1}{\log 2}$.

\section{A set containing all non-normal numbers}
\label{Sec_set_E}

Let $\beta>0$ be a small parameter that we will use to control the measure of a set $E$ which contains all non-normal numbers. 

For positive integers $N_b(m)$ and $N_{\text{CF}}(m,{\bm d})$ define
\begin{equation*}
E = \bigcup_{b \geq 2} \bigcup_{m \geq 1}  \bigcup_{N \geq N_b(m)}    \tilde{E}^c_b(1/m, N) 
\cup \bigcup_{{{\bm d}}} \bigcup_{m\geq 1}  \bigcup_{N \geq N_{CF}(m, {\bm d})+1} \tilde{E}^c_{CF}(1/m, {{\bm d}}, N)
\end{equation*}

The tilde shall indicate that we include for each interval of which the $E_b$ and $E_\text{CF}$ consist the two neighbouring intervals of the same lengths. This avoids that numbers starting with a `good' expansion e.g. in base $10$ and ending in all $9$'s lie outside $E$.

We further introduce a finitary  version of $E$. For $f$ from Proposition~\ref{prop_Omega} and any positive integer $k$ let
\begin{equation*}
E_k = \bigcup_{b = 2}^k \bigcup_{m = 1}^k \bigcup_{N = N_b(m)}^{kN_b}    \tilde{E}^c_b(1/m, N) 
\cup \bigcup_{m= 1}^k \bigcup_{{{\bm d}}, \vert {\bm d}\vert \leq k, d_i \leq k} \bigcup_{N = N_{CF}(m, {\bm d})+1}^{kN_{CF}} \tilde{E}^c_{CF}(1/m, {{\bm d}}, f(N), N).
\end{equation*}

Trivial upper bounds for the Lebesgue measure of $E$ and $E_k$ are
\begin{equation*}
\lambda(E) \leq 3 \sum_{b \geq 2} \sum_{m \geq 1} \sum_{N \geq N_b(m)}    \lambda(E^c_b(1/m, N)) 
+  \sum_{{{\bm d}}} \sum_{m\geq 1} \sum_{N \geq N_{CF}(m,{\bm d})+1} \lambda(E^c_{CF}(1/m, {{\bm d}}, f(N), N))
\end{equation*}
and
\begin{equation*}
\lambda(E_k) \leq 3 \sum_{b = 2}^k \sum_{m = 1}^k  \sum_{N = N_b(m)}^{kN_b}    \lambda(E^c_b(1/m, N)) 
+  \sum_{\substack{{\bm d}, \vert {{\bm d}} \vert \leq k,\\d_i \leq k, 1\leq i \leq k}} \sum_{m= 1}^k \sum_{N = N_{CF}(m,{\bm d})+1}^{kN_{CF}} \lambda(E^c_{CF}(1/m, {{\bm d}}, f(N), N)).
\end{equation*}

The starting lengths $N_b$ and $N_{\text{CF}}$ are chosen such that  $\lambda(E) \leq \beta$. In the integer case, they are allowed to depend on the base $b$ and $\epsilon=1/m$ and in the continued fraction case on $\epsilon=1/m$ and on the word ${{\bm d}}$. The function $f$ ensures computability of the set $E_k$ and its measure. Let $l$ be the length of the word ${\bm d}$.

\medskip

Let $r_k = \lambda(E\setminus E_k)$. 
It is clear that $r_k \rightarrow 0$ as $k \rightarrow \infty$. However, as the construction in Section~\ref{Sec_algorithm} depends on choosing suitable values for $k$, we give explicit upper bounds for $r_k$ in order for the construction to be completely deterministic.

\begin{Prop}\label{bads_are_small}
Choosing $N_b(m)$ and $N_{\text{CF}}(m,{\bm d})$ as indicated below, we have 
\begin{equation*}
\lambda(E) \leq \beta \quad \text{and} \quad r_k = O_\beta\left(\frac{1}{k}\right).
\end{equation*}
The implied constant in the estimate for $r_k$ can be derived explicitly from~\eqref{upper_bound_for_r_k} and depends only on $\beta$.
\end{Prop}

\begin{proof}
We show that for $N_b(m) = \frac{1}{2}C_1 b^4 m^3$ with $C_1 = \sqrt[3]{\frac{48}{\beta}}$
\begin{equation}\label{bad_integer_sum}
\sum_{b \geq 2} \sum_{m \geq 1}  \sum_{N \geq N_b(m)}    \lambda(E^c_b(1/m, N)) \leq \frac{\beta}{6},
\end{equation}
and that for $N_\text{CF}(m, {{\bm d}}) = C_2 900^8 \mu_G(\Delta_{\bm d})^{-8} l^8 m^6 d_l^2 \cdot \ldots \cdot d_1^2$ with $C_2 = \frac{384}{\beta}$
\begin{equation}\label{bad_cf_sum}
\sum_{{{\bm d}}} \sum_{m\geq 1} \sum_{N \geq N_\text{CF}(m,{\bm d})+1} \lambda(E^c_{CF}(1/m, {{\bm d}}, f(N), N)) \leq \frac{\beta}{6}.
\end{equation}

We treat sum~\eqref{bad_integer_sum} first. 
We have
\begin{align*}
 \sum_{b\geq 2} \sum_{m\geq 1} \sum_{N \geq N_b(m)} 2b e^{-\frac{2}{m^2 b^2}N}
= 2  \sum_{b\geq 2} b \sum_{m\geq 1} e^{-\frac{2}{m^2 b^2}N_b(m)} \frac{1}{1-e^{-2/(m^2b^2)}}.
\end{align*}
Note that $(1-e^{-2/(m^2 b^2)})^{-1} \leq 2m^2b^2$ for all $m\geq 1$, $b\geq 2$.
Hence this is
\begin{align}\label{blubb}
\leq 4  \sum_{b\geq 2} b^3 \sum_{m\geq 1} m^2 e^{-C_1 b^2m}.
\end{align}

Note that for $c > 0$ the function $x^2 e^{-cx}$ is strictly decaying for $x \geq \frac{2}{c}$. Thus, for $M_0 \geq \frac{2}{c}$, $\sum_{m = M_0+1}^\infty m^2 e^{-cm} \leq \int_{M_0}^\infty x^2 e^{-cx} dx$. 
Here we use that $\int_{0}^\infty x^2 e^{-cx} dx = \frac{2}{c^3}$ and that in our case $c = C_1 b^2 \geq 2$, so $\frac{2}{c} \leq 1$. Hence~\eqref{blubb} is
\begin{align*}
\leq 4  \sum_{b\geq 2} b^3 \sum_{m\geq 1} m^2 e^{-C_1 b^2m}
\leq 4  \sum_{b\geq 2} b^3 \frac{2}{C_1^3 b^6}
 = \frac{8}{C_1^3} \sum_{b\geq 2} \frac{1}{b^3}
<\frac{8}{C_1^3}.
\end{align*}
This is $\leq \frac{\beta}{6}$ for $C_1^3 \geq \frac{48}{\beta} > 1$.

\medskip

For continued fractions we use 
\begin{equation*}
\lambda(E^c_{CF}(1/m, {{\bm d}}, f(N), N)) \leq \lambda(E^c_{CF}(1/m, {{\bm d}}, N)) \leq e^{-\eta_{\text{CF}}(1/m,{{\bm d}})N^{1/2}}
\end{equation*} 
with $\eta$ from equation~\eqref{eta_cf_usable_bound}instead of the better bound $e^{-\eta_{\text{CF}}(1/m,{{\bm d}})  \frac{N}{\log N}}$ which is more difficult to work with.

We have $\frac{N}{\log N} \geq N^{1/2}$ for all $N\geq 1$ and that $e^{-\eta N^{1/2}}$ is strictly decaying for $N\geq 0$. Also note that for any $\eta>0$, $$\int_{x_0}^\infty e^{-\eta x^{1/2}} dx = \frac{2}{\eta^2} \frac{\eta x_0^{1/2}+1}{e^{\eta x_0^{1/2}}}.$$


We have
\begin{align*}
\eqref{bad_cf_sum} \leq \sum_{{{\bm d}}} \sum_{m\geq 1} \sum_{N \geq N_{\text{CF}}(m,{\bm d}) + 1} e^{-\eta(1/m,{\bm d}) N^{1/2}}
 \leq \sum_{{{\bm d}}} \sum_{m\geq 1} \frac{4}{\eta(1/m,{\bm d})} e^{-\frac{1}{2} \eta(1/m,{\bm d}) \sqrt{N_\text{CF}(m,{\bm d})}},
\end{align*}
where we used that $\eta N^{1/2}+1 \leq 2 \eta N^{1/2}$ for $N \geq \frac{1}{\eta^2}$. Put $N_\text{CF}(m,{{\bm d}}) = N_\text{CF}({{\bm d}}) m^6$ and set $\eta(1/m,{{\bm d}}) = \eta({{\bm d}})m^{-2}$ with $\eta({{\bm d}}) = (\mu_G(\Delta_{\bm d})/(900l))^2$ independent of $m$. We use again $\int_{0}^{\infty} x^2 e^{-cx}dx = \frac{2}{c^3}$. Here, $c = \frac{1}{2} \eta({{\bm d}}) \sqrt{N_\text{CF}({{\bm d}})}$ is larger than $2$ if $N_\text{CF}({{\bm d}}) \geq \frac{16}{\eta({{\bm d}})} = \frac{16 \cdot 900^2 l^2}{\mu_G(\Delta_{\bm d})^2}$. This is true by our choice of $N_\text{CF}(m,{{\bm d}})$. Thus
\begin{align*}
\sum_{{{\bm d}}} \sum_{m \geq 1} \frac{4m^2}{\eta({{\bm d}})} e^{-\frac{1}{2} \eta({{\bm d}}) \sqrt{N_\text{CF}({{\bm d}})} m} 
& \leq \sum_{{{\bm d}}} \frac{64}{\eta({{\bm d}})^4 N_\text{CF}({{\bm d}})^{3/2}}\\
&\leq 64 \sum_{l\geq 1} \sum_{d_l \geq 1} \cdots \sum_{d_1 \geq 1} \frac{1}{\eta({{\bm d}})^4 N_\text{CF}({{\bm d}})},
\end{align*}
where we split up the sum over all ${{\bm d}}$ into $\sum_{l\geq 1} \sum_{{{\bm d}}, \vert {{\bm d}} \vert = l} = \sum_{l\geq 1} \sum_{d_l \geq 1} \cdots \sum_{d_1 \geq 1}$.
Since $N_\text{CF}({{\bm d}}) = C_2 \frac{900^8}{\mu_G(\Delta_{\bm d})^8} l^8 d_l^2 \cdot \ldots \cdot d_1^2$, the previous term is
\begin{align}\label{foo}
&\leq \frac{64}{C_2} \sum_{l\geq 1} \sum_{d_l \geq 1} \cdots \sum_{d_1 \geq 1} \frac{1}{l^2 d_l^2 \cdot \ldots \cdot d_1^2 }.
\end{align}
Since $\sum_{n\geq 1} n^{-2} < 1$, this is just
\begin{equation*}
\leq \frac{64}{C_2},
\end{equation*}
which is $\leq \frac{\beta}{6}$ for $C_2 \geq \frac{384}{\beta}$.

\medskip

We continue showing that $r_k = O_\beta\left(\frac{1}{k}\right)$. Since $r_k = \lambda(E\setminus E_k)$, $r_k$ can be bounded above by the sum of an upper bound~\eqref{r_k_integer} for the integer part and an upper bound for the continued fraction part~\eqref{r_k_cf},
\begin{equation*}
r_k \leq \eqref{r_k_integer} + \eqref{r_k_cf}.
\end{equation*}

We tread the integer-base part first. 
\begin{equation}\label{r_k_integer}
= \left( \sum_{b=2}^k \sum_{m=1}^k \sum_{N=kN_b(m)}^\infty 
+ \sum_{b=2}^k \sum_{m=k+1}^\infty \sum_{N=N_b(m)}^\infty
+ \sum_{b=k+1}^\infty \sum_{m=1}^\infty \sum_{N=N_b(m)}^\infty \right) 2b e^{-\frac{2}{m^2}N}.
\end{equation}
Recall $N_b(m)= \frac{1}{2} C_1 b^4 m^3$ with $C_1 = \sqrt[3]{\frac{48}{\beta}}$. 
We have for the first sum in~\eqref{r_k_integer},
\begin{align*}
\sum_{b=2}^k \sum_{m=1}^k \sum_{N=kN_b(m) +1}^\infty 2b e^{-\frac{2}{m^2 b^2}N} 
\leq 4 \sum_{b=2}^k b \sum_{m=1}^k m^2  e^{-\frac{2}{m^2b^2}kN_b(m)}
\leq \frac{8}{C_1^3} \sum_{b=2}^k b^3 \frac{1}{b^6 k^3}
\leq \frac{8}{C_1^3 k^3}.
\end{align*}
For the second sum in~\eqref{r_k_integer},
\begin{align*}
\sum_{b=2}^k \sum_{m=k+1}^\infty \sum_{N=N_b(m)}^\infty 2b e^{-\frac{2}{m^2}N} 
&\leq 4 \sum_{b=2}^k b^3 \sum_{m=k+1}^\infty m^2  e^{-\frac{2}{m^2b^2}kN_b(m)}\\
&\leq 4 \sum_{b=2}^k \left(\frac{kb}{C_1} + \frac{2}{C_1^2 kb} + \frac{2}{C_1^3 k^3 b^3} \right) e^{-C_1b^2k^2}\\
&\leq 20 k \sum_{b=2}^k e^{-C_1 b^2 k^2} 
\leq 40 k e^{-C_1 k^2},
\end{align*}
where we used that $\int_{x_0}^\infty x^2 e^{-cx} dx = (x_0^2/c + 2 x_0/c^2 + 2/c^3)e^{-cx_0}$. Here again  $x^2 e^{-cx}$ is strictly decreasing for $x \geq \frac{2}{c}$. In our case $c = k b^2 C_1 \geq 2$ so that $\frac{2}{c} \leq 1$.

For the third sum in~\eqref{r_k_integer},
\begin{align*}
\sum_{b=k+1}^\infty \sum_{m=1}^\infty \sum_{N=N_b(m)}^\infty  2b e^{-\frac{2}{m^2 b^2}N}
\leq 4 \sum_{b=k+1}^\infty b^3 \sum_{m=1}^\infty m^2  e^{-\frac{2}{m^2b^2}N_b(m)}
\leq  4  \sum_{b=k+1}^\infty b^3 \frac{2}{C_1^3 b^6}
<\frac{8}{C_1^3 k^2}.
\end{align*}

Thus
\begin{equation}\label{bound_r_k_integer_part}
\eqref{r_k_integer} \leq \frac{8}{C_1^3 k^3} + 40 k e^{-C_1 k^2} + \frac{8}{C_1^3 k^2}.
\end{equation}

The continued fraction part of $r_k$ can be bounded above by
\begin{equation}\label{r_k_cf}
= \eqref{r_k_cf_1} + \eqref{r_k_cf_2} + \eqref{r_k_cf_3},
\end{equation}
where
\begin{equation}\label{r_k_cf_1}
=\sum_{\substack{{\bm d}, \vert {{\bm d}} \vert \leq k,\\d_i \leq k, 1\leq i \leq k}} \sum_{m= 1}^k \sum_{N = kN_\text{CF}(m,{\bm d})+1}^{\infty} e^{-\frac{1}{2} \eta(1/m,{{\bm d}}) \sqrt{N_\text{CF}(m,{\bm d})}},
\end{equation}
\begin{equation}\label{r_k_cf_2}
=\sum_{\substack{{\bm d}, \vert {{\bm d}} \vert \leq k,\\d_i \leq k, 1\leq i \leq k}} \sum_{m= k+1}^\infty \sum_{N = N_\text{CF}(m,{\bm d})+1}^{\infty} e^{-\frac{1}{2} \eta(1/m,{\bm d}) \sqrt{N_\text{CF}(m,{\bm d})}},
\end{equation}
and
\begin{equation}\label{r_k_cf_3}
=\sum_{\substack{{\bm d}, \vert {{\bm d}} \vert \leq k, \exists 1 \leq i \leq k : d_i \geq k+1,\\ \text{or } {\bm d}, \vert {{\bm d}} \vert \geq k+1}} \sum_{m= 1}^\infty \sum_{N = N_\text{CF}(m,{\bm d})+1}^\infty e^{-\frac{1}{2} \eta(1/m,{\bm d}) \sqrt{N_\text{CF}(m,{\bm d})}}.
\end{equation}


The sum~\eqref{r_k_cf_1} decays linearly in $k$ with constant $C_2$ replaced by $kC_2$. 

The sum of~\eqref{r_k_cf_2} requires some care. As before, we have
\begin{align*}
\eqref{r_k_cf_2} \leq \sum_{ \text{all } {{\bm d}}} \sum_{m \geq k+1} \frac{4m^2}{\eta({{\bm d}})} e^{-\frac{1}{2} \eta({{\bm d}}) \sqrt{N_\text{CF}({{\bm d}})} m}.
\end{align*}

For $x >\frac{16}{c^2}$, $x^2 e^{-cx} \leq e^{-cx/2}$. For $c\geq 8$, $\frac{2}{c} \geq \frac{16}{c^2}$ and $\frac{2}{c} < 1$. Hence for $c\geq 8$ and all $M_0$,
\begin{equation}\label{lemma_sum_x_squared_e_cx}
\sum_{m = M_0+1}^\infty m^2 e^{-cm} \leq \int_{M_0}^\infty x^2 e^{-cx} dx \leq \int_{M_0}^\infty e^{-\frac{c}{2}x} dx = \frac{2}{c} e^{-\frac{c}{2}M_0}.
\end{equation}
Here $c = \frac{1}{2} \eta({{\bm d}}) \sqrt{N_\text{CF}({{\bm d}})} \geq 8$ by the choice of $N_\text{CF}({{\bm d}})$.

Thus
\begin{equation*}
\eqref{r_k_cf_2} \leq 4 \sum_{{{\bm d}}} \frac{4}{\eta({{\bm d}})^2 \sqrt{N_\text{CF}({{\bm d}})}} e^{-\frac{1}{4} \eta({{\bm d}}) \sqrt{N_\text{CF}({{\bm d}})} k}.
\end{equation*}
Here, $\eta({{\bm d}})^2 \sqrt{N_\text{CF}({{\bm d}})} = \sqrt{C_2} d_l \cdot \ldots \cdot d_1 \geq \sqrt{C_2}$ and $\frac{1}{4} \eta({{\bm d}}) \sqrt{N_\text{CF}({{\bm d}})} = \frac{1}{4}  \sqrt{C_2} 900^2 \mu_G(\Delta_{\bm d})^2 l^2 d_l \cdot \ldots \cdot d_1 \geq l^2 d_l \cdot \ldots \cdot d_1$.
Thus
\begin{align*}
\eqref{r_k_cf_2} & \leq \frac{16}{\sqrt{C_2}} \sum_{l \geq 1} \sum_{d_l \geq 1} \cdots \sum_{d_1\geq 1} e^{-l^2 d_l \cdot \ldots \cdot d_1 k}\\
& \leq  \frac{16}{\sqrt{C_2}} \sum_{l \geq 1} 2^l e^{-l^2 k} \\
&\leq \frac{16}{\sqrt{C_2}} \left( e^{1-k} + \sum_{l\geq 2} e^{-lk} \right) \\
&\leq \frac{16(2+e)}{\sqrt{C_2}} e^{-k}.
\end{align*}

As in~\eqref{foo}, the sum~\eqref{r_k_cf_3} over the restricted range of words ${{\bm d}}$ can be bounded above by
\begin{align*}
\eqref{r_k_cf_3} \leq \eqref{bla1} + \eqref{bla2},
\end{align*}
where
\begin{equation}\label{bla1}
= \frac{64}{C_2} \sum_{l \geq k+1} \sum_{\vert {{\bm d}} \vert =l} \frac{1}{l^2 d_l^2 \cdot \ldots \cdot d_1^2},
\end{equation}
and
\begin{equation}\label{bla2}
= \frac{64}{C_2} \sum_{1 \leq  l \leq k}
\left( \sum_{d_l \geq k+1} \sum_{\substack{d_i \geq 1,\\1\leq i \leq l-1}}
+ \sum_{1 \leq d_l \leq k} \sum_{d_{l-1} \geq k+1} \sum_{\substack{d_i \geq 1,\\1\leq i \leq l-2}}
+ \ldots 
+ \sum_{\substack{1 \leq d_i \leq k,\\2 \leq i \leq l}} \sum_{d_1 \geq k+1} \right)  \frac{1}{l^2 d_l^2 \cdot \ldots \cdot d_1^2}.
\end{equation}

 Any sum over an unrestricted range of $d_i \geq 1$ gives a convergent term less than $1$. The sums over the restricted range $d_i \geq k+1$ are bounded above by $\frac{1}{k}$. Finally, the restricted ranges $1\leq d_i \leq k$ contribute at most $(\frac{\pi^2}{6}-1 - \frac{1}{k+1})$ which is less than $0.7$. Thus~\eqref{r_k_cf_3} can be bounded above by
\begin{equation*}
\leq \frac{64}{C_2} \left( \frac{1}{k} 
+ \frac{1}{k}\sum_{1 \leq l \leq k} \frac{1}{l^2}
\left( \sum_{i=0}^{l-1} 0.7^i \right) \right).
\end{equation*}
This expression converges and is
\begin{equation*}
\leq \frac{214}{C_2k}.
\end{equation*}

To conclude, we have shown
\begin{equation*}
\eqref{r_k_cf} \leq \frac{64}{C_2 k} + \frac{16(2+e)}{\sqrt{C_2}} e^{-k} + \frac{214}{C_2k},
\end{equation*}
which together with~\eqref{bound_r_k_integer_part} implies
\begin{equation}\label{upper_bound_for_r_k}
r_k \leq \frac{8}{C_1^3 k^3} + 40 k e^{-C_1 k^2} + \frac{8}{C_1^3 k^2} + \frac{64}{C_2 k} + \frac{16(2+e)}{\sqrt{C_2}} e^{-k} + \frac{214}{C_2k}.
\end{equation}
\end{proof}

\section{Restricting partial quotients}
\label{Sec_Omega}

Fix $f \colon \mathbb{N} \rightarrow \mathbb{N}$ and denote
\begin{align*}
\Omega_N &= \{x \in [0,1) \mid a_i(x) \leq f(i), 1\leq i \leq N\}\\
\Omega &= \bigcap_{N\geq 1} \Omega_N = \{x \in [0,1) \mid a_i(x) \leq f(i),  i \geq 1\}.
\end{align*}
By appropriately choosing $f$, $\Omega$ has measure arbitrarily close to $1$.

\begin{Prop}\label{prop_Omega}
Let $f(i) = A \cdot 2^{i}-2$ with a positive integer $A\geq 3$. Then
\begin{align*}
\lambda(\Omega) & \geq 1 - \frac{2}{A} > 0, \quad \text{ and } \quad \lambda(\Omega_N \setminus \Omega) \leq \frac{1}{A} \frac{1}{2^{N+1}}.
\end{align*}
\end{Prop}

\begin{proof}
Since $\log(2)\mu_G \leq \lambda \leq 2\log(2)\mu_G$ and the invariance of $\mu_G$ under the Gauss map we have
\begin{align*}
\lambda(\Omega) & = \lambda \{x \in [0,1) : a_i(x) \leq f(i), i \geq 1\} \\
&= 1 - \lambda\left( \bigcup_{i \geq 1} \{x\in [0,1) :  a_i(x) \geq f(i) +1 \} \right) \\
& \geq 1 - 2\log(2) \sum_{i=1}^\infty \mu_G \{x \in [0,1) : a_i(x) \geq  f(i)+1 \} \\
 &= 1 - 2 \log(2) \sum_{i=1}^\infty \mu_G \{x \in [0,1) : a_1(x) \geq f(i) +1\} \\
& \geq 1 - 2 \sum_{i=1}^\infty \lambda \{x \in [0,1) : a_1(x) \geq f(i) +1 \} \\
 &= 1 - 2 \sum_{i=1}^\infty \frac{1}{f(i) + 2}.
\end{align*}

With $f(i) = A 2^{i} - 2$, $\lambda(\Omega) \geq 1 - \frac{2}{A}.$

\medskip
For the second assertion, 
\begin{align*}
\Omega_N \setminus \Omega & = \bigcap_{1 \leq i \leq N} \{x \in [0,1) : a_i(x) \leq f(i) \} \cap \bigcap_{i \geq N+1} \{x \in [0,1) : a_i(x) \geq f(i)+1 \}.
\end{align*}
Thus
\begin{align*}
\lambda(\Omega_N \setminus \Omega) &\leq \lambda(\{x \in [0,1) : a_{N+1}(x) \geq f(N+1)+1 \})\\
&= \frac{1}{A2^{N+1}},
\end{align*}
since the measure of the intersection of a number of sets can be trivially bounded above by the measure of one of the intersecting sets.
\end{proof}

Denote $\omega = \frac{2}{A}$ and $\omega_N = \frac{1}{A} \frac{1}{2^{N+1}}$ so that $\lambda(\Omega) \geq 1- \omega$ and $\lambda(\Omega_N\setminus \Omega) \leq \omega_N$. Since $A \geq 3$, $1 - \omega > 0$ and as $N$ tends to infinity, $\omega_N$ tends to zero.

Note that $\Omega_N$ in $[0,1)$ is a union of cylinder intervals with rational endpoints and is thus computable, as well as its Lebesgue measure $\lambda(\Omega_N)$.

\section{Set-theoretic Lemmas}
\label{sec_set_theoretic_lemmas}
In the following, let $c \subset [0,1)$ be an interval and $M<N$, $k < l$ positive integers and $\Omega$, $\Omega_N$, $E$, $E_k$, $r_k$, $\omega$ and $\omega_N$ as before.

\begin{Lem}
We have
\begin{gather}
\lambda(E_l \setminus E_k) \leq r_k,\\
\lambda((\Omega \setminus E) \cap c) \geq
\lambda((\Omega \setminus E_k) \cap c) - r_k,\\
\lambda((\Omega_N \setminus E) \cap c) \geq
\lambda((\Omega_N \setminus E_k) \cap c) - r_k,\\
\lambda((\Omega \setminus E_l) \cap c) \geq \lambda((\Omega \setminus E_k) \cap c) - r_k,\\
\lambda((\Omega_N \setminus E_l) \cap c) \geq \lambda((\Omega_N \setminus E_k) \cap c) - r_k.
\end{gather}
\end{Lem}

\begin{proof}
$\lambda(E_l \setminus E_k) \leq r_k$ follows from $E_l \setminus E_k \subset E \setminus E_k$.

We have
\begin{align*}
(\Omega \setminus E) \cap c 
&= ((\Omega \setminus E_k) \cap c) \setminus ((E \setminus E_k) \cap c).
\end{align*}
Hence
\begin{align*}
\lambda((\Omega \setminus E) \cap c) 
&\geq \lambda((\Omega \setminus E_k) \cap c) - \lambda( E \setminus E_k \cap c) \\
&\geq \lambda((\Omega \setminus E_k) \cap c) - \lambda(E \setminus E_k) \\
&\geq \lambda((\Omega \setminus E_k) \cap c) - r_k.
\end{align*}
The same argument works with $\Omega$ replaced by $\Omega_N$ and $E$ replaced by $E_l$ which gives the remaining inequalities.
\end{proof}

\begin{Lem}
We have
\begin{gather*}
\lambda( (\Omega \setminus E_k) \cap c) \geq \lambda( (\Omega_N \setminus E_k) \cap c) -  \omega_N,\\
\lambda( (\Omega_N \setminus E_k) \cap c) \geq \lambda( (\Omega_M \setminus E_k) \cap c) -  \omega_M.
\end{gather*}
\end{Lem}

\begin{proof}
\begin{align*}
(\Omega_N \setminus E_k) \cap c &= ( (\Omega \sqcup(\Omega_N \setminus \Omega)) \setminus E_k ) \cap c \\
&= ( \Omega \setminus E_k \sqcup (\Omega_N \setminus \Omega) \setminus E_k) \cap c \\
&= (\Omega \setminus E_k) \cap c \sqcup (\Omega_N \setminus \Omega) \setminus E_k \cap c.
\end{align*}
Consequently,
\begin{align*}
\lambda((\Omega \setminus E_k) \cap c) &= \lambda((\Omega_N \setminus E_k) \cap c) - \lambda( (\Omega_N \setminus \Omega) \setminus E_k \cap c) \\
&\geq \lambda((\Omega_N \setminus E_k) \cap c) - \lambda(\Omega_N \setminus \Omega)\\
&\geq \lambda((\Omega_N \setminus E_k) \cap c) - \omega_N.
\end{align*}
The second inequality follows using the same argument applied to $\Omega_M = \Omega_N \sqcup \Omega_M \setminus \Omega_N$.
\end{proof}

\section{Algorithm}
\label{Sec_algorithm}

%
%

Let $\beta > 0$ such that $1-\omega - \beta > 0$.

\subsection{First (binary) digit}

We choose $N_1$ and $k_1$ such that 
\begin{equation*}
\frac{1}{2}(1 - \omega - \beta) -  \omega_{N_1} - r_{k_1} \geq \frac{1}{4}(1 - \omega - \beta) > 0
\end{equation*}
This can be achieved for example with $N_1$ and $k_1$ such that $ \omega_{N_1} \leq \frac{1}{8}(1-\omega-\beta)$ and $r_{k_1} \leq \frac{1}{8}(1-\omega-\beta)$. Suitable values for $N_1$ and $k_1$ are computable using Propositions~\ref{prop_Omega} and~\ref{bads_are_small}.

We have 
\begin{align*}
\lambda((\Omega_{N_1} \setminus E_{k_1}) \cap [0,1/2)) \ +\  \lambda((\Omega_{N_1} \setminus E_{k_1}) \cap [1/2,1)) 
& = \lambda(\Omega_{N_1} \setminus E_{k_1}) \\
& \geq \lambda(\Omega_{N_1}) - \lambda(E_{k_1}) \\
& \geq 1 - \omega - \beta,
\end{align*}
which is $> 0$ by assumption on $\beta$. The last lower bound is independent of $N_1$ and $k_1$ because $\lambda(\Omega_N) \geq \lambda(\Omega)\geq 1- \omega$ for any $N$ and $\lambda(E_k) \leq \lambda(E) < \beta$ for any $k$.

Hence there is an interval $c_1 \in \{[0,1/2), [1/2,1)\}$ such that 
\begin{equation*}
\lambda((\Omega_{N_1} \setminus E_{k_1}) \cap c_1) \geq \frac{1}{2}(1 - \omega - \beta) > 0.
\end{equation*}
Since the Lebesgue measure of $(\Omega_{N_1} \setminus E_{k_1}) \cap c_1$ can be computed, the interval $c_1$ can be computably obtained.

We have
\begin{align*}
\lambda((\Omega \setminus E) \cap c_1) & \geq \lambda((\Omega \setminus E_{k_1}) \cap c_1) - r_{k_1} \\
&\geq \lambda((\Omega_{N_1} \setminus E_{k_1}) \cap c_1) -  \omega_{N_1} - r_{k_1}\\
& \geq \frac{1}{2}(1 - \omega - \beta) -  \omega_{N_1} - r_{k_1}.
\end{align*}
Hence $\lambda((\Omega \setminus E) \cap c_1) > 0$, so there are numbers in $\Omega\cap c_1$ outside $E$, i.e. whose first binary digit is given by $c_1$.

\subsection{Second digit}

Let $N_2$ and $k_2$ be such that $\epsilon_{N_2} \leq \frac{1}{32}(1-\omega-\beta)$ and $r_{k_2} \leq \frac{1}{32}(1-\omega-\beta)$.

We have
\begin{align*}
\lambda((\Omega_{N_2} \setminus E_{k_2}) \cap c_2^1) 
+ \lambda((\Omega_{N_2} \setminus E_{k_2}) \cap c_2^2)
&= \lambda((\Omega_{N_2} \setminus E_{k_2}) \cap c_1) \\
&\geq \lambda((\Omega_{N_2} \setminus E_{k_1}) \cap c_1) - r_{k_1} \\
&\geq \lambda((\Omega_{N_1} \setminus E_{k_1}) \cap c_1) -\omega_{N_1} - r_{k_1} \\
&\geq \frac{1}{4}(1-\omega - \beta)\\
&> 0
\end{align*}
by the choice of $N_1$ and $k_1$ from step 1. Hence one half $c_2$ of $c_1$ satisfies
\begin{equation*}
\lambda((\Omega_{N_2} \setminus E_{k_2}) \cap c_2) \geq \frac{1}{8}(1-\omega-\beta) > 0.
\end{equation*}
Which half of $c_1$ to choose can be computed.

Finally, we have
\begin{align*}
\lambda((\Omega \setminus E) \cap c_2)
&\geq \lambda((\Omega \setminus E_{k_2}) \cap c_2) - r_{k_2} \\
&\geq \lambda((\Omega_{N_2} \setminus E_{k_2}) \cap c_2) - \omega_{N_2} - r_{k_2}\\
&\geq \frac{1}{8}(1-\omega - \beta) - \omega_{N_2} - r_{k_2} \\
&\geq \frac{1}{16}(1- \omega - \beta)\\
&> 0,
\end{align*}
hence there are numbers in $\Omega\cap c_2$ outside $E$, i.e. whose binary expansion starts with digits given by $c_1$, $c_2$.

This algorithm produces the binary digits of a real number $\nu$. 

\begin{Thm}
\label{BIGTHEOREM}
The number $\nu$ is computable. It is furthermore absolutely normal and continued fraction normal.
\end{Thm}

\begin{proof}
All values $N_i$, $k_i$, $\omega_{N_i}$, $r_{k_i}$ and all appearing measures can be computed, hence $\nu$ is computable.

Suppose $\nu$ was not absolutely normal and continued fraction normal. Then $\nu$ is an element of $E$, i.e. $\nu$ is contained in an interval $I \in E$ of positive measure. Since $\nu$ by construction lies in all $c_i$, for some $i$ we have $c_i \subset I$, hence $c_i \subset E$. This implies that $(\Omega \setminus E) \cap c_i = \emptyset$, a contradiction since we chose $c_i$ to be such that $\lambda((\Omega \setminus E) \cap c_i) > 0$.
\end{proof}

Note that we implicitly used that absolute normality is equivalent to simple normality to all bases $b\geq 2$ (see e.g.~\cite[Ch.~4]{bugeaud2012distribution}).

\medskip
\textbf{Acknowledgements.}
The author was supported by the Austrian Science Fund (FWF): I 1751-N26; W1230, Doctoral Program ``Discrete Mathematics''; and  SFB F 5510-N26. He would like to thank the Erwin-Schr\"odinger-Institute in Vienna for its hospitality.



\begin{bibdiv}
\begin{biblist}


\bib{ADLER198195}{article}{
title = {A construction of a normal number for the continued fraction transformation},
journal = {Journal of Number Theory},
volume = {13},
number = {1},
pages = {95 - 105},
year = {1981},
issn = {0022-314X},
doi = {http://dx.doi.org/10.1016/0022-314X(81)90031-7},
author = {Adler, Roy},  
author={Keane, Michael},
author={Smorodinsky, Meir}
}

\bib{alvarez_becher2015:levin}{article}{
	AUTHOR = {Alvarez, Nicol{\'a}s},	
	AUTHOR = {Becher, Ver{\'o}nica},
	TITLE = {M. Levin's construction of absolutely normal numbers with very low discrepancy},
	JOURNAL = {arXiv:1510.02004},
	URL = {http://arxiv.org/abs/1510.02004},
}

%
%
\bib{becher_figueira:2002}{article}{
	AUTHOR = {Becher, Ver{\'o}nica},
    AUTHOR = {Figueira, Santiago},
    TITLE  = {An example of a computable absolutely normal number},
    JOURNAL = {Theoretical Computer Science},
    VOLUME = {270},
    YEAR = {2002},
    PAGES = {126--138},
}
%
\bib{becher_figueira_picchi2007:turing_unpublished}{article}{
	AUTHOR = {Becher, Ver{\'o}nica},
    AUTHOR = {Figueira, Santiago},
    AUTHOR = {Picchi, Rafael},
     TITLE = {Turing's unpublished algorithm for normal numbers},
   JOURNAL = {Theoret. Comput. Sci.},
  FJOURNAL = {Theoretical Computer Science},
    VOLUME = {377},
      YEAR = {2007},
    NUMBER = {1-3},
     PAGES = {126--138},
      ISSN = {0304-3975},
     CODEN = {TCSDI},
   MRCLASS = {03D80 (01A60 03-03 11K16 11Y16 68Q30)},
  MRNUMBER = {2323391 (2008j:03064)},
MRREVIEWER = {George Barmpalias},
       DOI = {10.1016/j.tcs.2007.02.022},
       URL = {http://dx.doi.org/10.1016/j.tcs.2007.02.022},
}
%
\bib{becher_heiber_slaman2013:polynomial_time_algorithm}{article}{
      author={Becher, Ver{\'o}nica},
      author={Heiber, Pablo~Ariel},
      author={Slaman, Theodore~A.},
       title={A polynomial-time algorithm for computing absolutely normal
  numbers},
        date={2013},
        ISSN={0890-5401},
     journal={Inform. and Comput.},
      volume={232},
       pages={1\ndash 9},
         url={http://dx.doi.org/10.1016/j.ic.2013.08.013},
      review={\MR{3132518}},
}
%

%
\bib{besicovitch1935:epsilon}{article}{
author={Besicovitch, A. S.},
title={The asymptotic distribution of the numerals in the decimal representation of the squares of the natural numbers},
year={1935},
pages={146–156},
journal={Math. Zeit.},
number={39},
}

\bib{borel1909}{article}{
year={1909},
issn={0009-725X},
journal={Rendiconti del Circolo Matematico di Palermo},
volume={27},
number={1},
doi={10.1007/BF03019651},
title={Les probabilit\'{e}s d\'{e}nombrables et leurs applications arithm\'{e}tiques},
url={http://dx.doi.org/10.1007/BF03019651},
publisher={Springer-Verlag},
author={Borel, \'{E}mile},
pages={247-271},
language={French}
}
%
\bib{bradley}{article}{
author = {Bradley, Richard C.},
doi = {10.1214/154957805100000104},
fjournal = {Probability Surveys},
journal = {Probab. Surveys},
pages = {107--144},
publisher = {The Institute of Mathematical Statistics and the Bernoulli Society},
title = {Basic Properties of Strong Mixing Conditions. A Survey and Some Open Questions},
url = {http://dx.doi.org/10.1214/154957805100000104},
volume = {2},
year = {2005}
}



%
\bib{bugeaud2012distribution}{book}{
  title={Distribution Modulo One and Diophantine Approximation},
  author={Bugeaud, Y.},
  isbn={9780521111690},
  lccn={2012013417},
  series={Cambridge Tracts in Mathematics},
  year={2012},
  publisher={Cambridge University Press},
}
%
%

%
%
\bib{drmota_tichy}{book}{
    AUTHOR = {Drmota, Michael},
    AUTHOR = {Tichy, Robert F.},
     TITLE = {Sequences, discrepancies and applications},
    SERIES = {Lecture Notes in Mathematics},
    VOLUME = {1651},
 PUBLISHER = {Springer-Verlag, Berlin},
      YEAR = {1997},
     PAGES = {xiv+503},
      ISBN = {3-540-62606-9},
   MRCLASS = {11Kxx (11K06 11K38)},
  MRNUMBER = {1470456 (98j:11057)},
MRREVIEWER = {Oto Strauch},
}


\bib{figueira_nies}{article}{
author={S. Figueira},
author={A. Nies},
title={Feasible analysis and randomness},
journal={Manuscript},
year={2013}
}


\bib{iosifescu_kraaikamp}{book}{
author={M. Iosifescu},
author={Kraaikamp, Cornelis},
title={Metrical Theory of Continued Fractions},
year={2002},
publisher={Springer Netherlands},
note={Mathematics and Its Applications}
}


\bib{kessebohmer_schindler}{article}{
title={Limit theorems for counting large continued fraction digits},
author={Kesseb\"ohmer, Marc},
author={Schindler, Tanja},
journal={arXiv:1604.06612}
}

\bib{kuipers_niederreiter}{book}{
    AUTHOR = {Kuipers, L.},
    AUTHOR = {Niederreiter, H.},
     TITLE = {Uniform distribution of sequences},
      NOTE = {Pure and Applied Mathematics},
 PUBLISHER = {Wiley-Interscience [John Wiley \& Sons], New
              York-London-Sydney},
      YEAR = {1974},
     PAGES = {xiv+390},
   MRCLASS = {10K05 (22D99)},
  MRNUMBER = {0419394 (54 \#7415)},
MRREVIEWER = {P. Gerl},
}


\bib{MR1504765}{article}{
   author={Lebesgue, H.},
   title={Sur certaines d\'emonstrations d'existence},
   language={French},
   journal={Bull. Soc. Math. France},
   volume={45},
   date={1917},
   pages={132--144},
   issn={0037-9484},
   review={\MR{1504765}},
}


\bib{levin1979:absolutely_normal}{article}{
    AUTHOR = {Levin, M. B.},
     TITLE = {Absolutely normal numbers},
   JOURNAL = {Vestnik Moskov. Univ. Ser. I Mat. Mekh.},
  FJOURNAL = {Vestnik Moskovskogo Universiteta. Seriya I. Matematika,
              Mekhanika},
      YEAR = {1979},
    NUMBER = {1},
     PAGES = {31--37, 87},
      ISSN = {0201-7385},
   MRCLASS = {10K25},
  MRNUMBER = {525299 (80d:10076)},
MRREVIEWER = {J. Galambos},
}
%



\bib{lutz_mayordomo_nearly_linear_time}{article}{
	AUTHOR = {Lutz, Jack H.},	
	AUTHOR = {Mayordomo, Elvira},
	TITLE = {Computing Absolutely Normal Numbers in Nearly Linear Time},
	JOURNAL = {arXiv:1611.05911},
	URL = {http://arxiv.org/abs/1611.05911},
}

%

\bib{Madritsch2016}{article}{
author={Madritsch, Manfred G.},
author={Mance, Bill},
title={Construction of $\mu$-normal sequences},
journal={Monatshefte f{\"u}r Mathematik},
year={2016},
volume={179},
number={2},
pages={259--280},
issn={1436-5081},
doi={10.1007/s00605-015-0837-1},
url={http://dx.doi.org/10.1007/s00605-015-0837-1}
}



\bib{mayordomo}{article}{
author={E. Mayordomo},
title={Construction of an absolutely normal real number in polynomial},
journal={Manuscript},
year={2013}
}


\bib{merlevede_peligrad_rio}{collection}{
author = {Merlev\`{e}de, Florence},
author={Peligrad, Magda},
author={Rio, Emmanuel},
edition = {High Dimensional Probability V: The Luminy Volume, 273--292},
doi = {10.1214/09-IMSCOLL518},
series={IMS Collections},
title = {Bernstein inequality and moderate deviations under strong mixing conditions},
volume = {5},
year = {2009},
}

\bib{Philipp1988}{article}{
author={Philipp, Walter},
title={Limit theorems for sums of partial quotients of continued fractions},
journal={Monatshefte f{\"u}r Mathematik},
year={1988},
volume={105},
number={3},
pages={195--206},
issn={1436-5081},
doi={10.1007/BF01636928}
}


\bib{MR0101857}{article}{
   author={Postnikov, A. G.},
   author={Pyatecki{\u\i}, I. I.},
   title={A Markov-sequence of symbols and a normal continued fraction},
   language={Russian},
   journal={Izv. Akad. Nauk SSSR. Ser. Mat.},
   volume={21},
   date={1957},
   pages={729--746},
   issn={0373-2436},
   review={\MR{0101857}},
}


\bib{queffelec}{collection}{
author = {Queff\'elec, Martine},
edition = {Dynamics \& Stochastics, 225--236},
doi = {10.1214/074921706000000248},
series = {IMS Lecture Notes--Monograph Series},
title = {Old and new results on normality},
volume = {48},
year = {2006}
}



\bib{scheerer2015:schmidt}{article}{
author = {{Scheerer}, A.-M.},
    title = {Computable Absolutely Normal Numbers and Discrepancies},
     year = {2015},   
     JOURNAL = {arXiv:1511.03582},
	URL = {http://arxiv.org/abs/1511.03582},
}
%

%

%
\bib{schmidt1961:uber_die_normalitat}{article}{
    AUTHOR = {Schmidt, Wolfgang M.},
     TITLE = {\"{U}ber die {N}ormalit\"at von {Z}ahlen zu verschiedenen
              {B}asen},
   JOURNAL = {Acta Arith.},
  FJOURNAL = {Polska Akademia Nauk. Instytut Matematyczny. Acta Arithmetica},
    VOLUME = {7},
      YEAR = {1961/1962},
     PAGES = {299--309},
      ISSN = {0065-1036},
   MRCLASS = {10.33},
  MRNUMBER = {0140482 (25 \#3902)},
MRREVIEWER = {N. G. de Bruijn},
}
%

%
\bib{siegel:towards_a_usable_theory_of_chernoff}{article}{
    author = {Alan Siegel},
    title = {Toward a usable theory of Chernoff Bounds for heterogeneous and partially dependent random variables},
    year = {1992}
}
%
\bib{sierpinski1917:borel_elementaire}{article}{
	    AUTHOR = {Sierpinski, Waclaw},
     TITLE = {D\'emonstration \'el\'ementaire du th\'eor\`eme de M. Borel sur les nombres absolument normaux et d\'etermination effective d'un tel nombre},
   JOURNAL = {Bulletin de la Soci\'et\'e Math\'ematique de France},
    VOLUME = {45},
      YEAR = {1917},
     PAGES = {127--132},
}
%
\bib{turing1992:collected_works}{article}{
    AUTHOR = {Turing, Alan},
     TITLE = {A Note on Normal Numbers},
     BOOKTITLE = {Collected Works of A. M. Turing, Pure Mathematics, edited by J. L. Britton},
   PUBLISHER = {North Holland},
      YEAR = {1992},
     PAGES = {117--119},
}
%

\bib{Vandehey2016424}{article}{
title = {New normality constructions for continued fraction expansions},
journal = {Journal of Number Theory},
volume = {166},
number = {},
pages = {424 - 451},
year = {2016},
issn = {0022-314X},
doi = {http://dx.doi.org/10.1016/j.jnt.2016.01.030},
author = {Vandehey, Joseph}
}




\end{biblist}
\end{bibdiv}


\end{document}